\documentclass[12pt,a4paper,oneside]{article}
\usepackage[a4paper,left=3cm,right=3cm, top=3cm, bottom=3cm]{geometry}
\usepackage[latin1]{inputenc}
\usepackage{enumerate}
\usepackage{mathrsfs}
\usepackage[T1]{fontenc}
\usepackage {graphicx}
\usepackage{amsmath}
\usepackage{amsthm}
\usepackage{amssymb}
\usepackage{hyperref}
\theoremstyle{plain}
\newtheorem{thm}{Theorem}[section] 
\newtheorem{cor}[thm]{Corollary}
\newtheorem{lem}[thm]{Lemma}
\newtheorem{prop}[thm]{Proposition}
\theoremstyle{definition}
\newtheorem{defn}[thm]{Definition}

\theoremstyle{remark}

\newcommand{\R}{\mathbb{R}}

\newcommand{\C}{\mathbb{C}} 
\newcommand{\vertiii}[1]{{\left\vert\kern-0.25ex\left\vert\kern-0.25ex\left\vert #1 
    \right\vert\kern-0.25ex\right\vert\kern-0.25ex\right\vert}}
\title{On the class of continuous images of non-commutative Valdivia compacta\footnote{Research was supported in part by the 
Universit\`a degli Studi of Milano (Italy) and in part by the Gruppo Nazionale per l'Analisi Matematica, la Probabilit\`a e le loro Applicazioni (GNAMPA) of the Istituto Nazionale di Alta Matematica (INdAM) of Italy.}}
\author{Jacopo Somaglia}
\date{}
\begin{document}
\maketitle
\begin{abstract}
\noindent
We investigate the class of continuous images of non-commutative Valdivia compact spaces, in particular its subclass of weakly non-commutative Corson countably compact spaces. A key tool is the study of non-commutative Corson countably compact spaces and their stability. The results are the non-commutative version of results by O. Kalenda (2003). Moreover, we present a study of retractional skeletons on Aleksandrov duplicates of ordinal spaces.\\


 
\end{abstract}
\section{Introduction}

In order to investigate structural properties of certain topological and Banach spaces, it is often convenient to define special families of retractions on them. For example Amir and Lindenstrauss used projectional resolution of identity (PRI) to characterize Eberlein compact spaces \cite{AmirLin}.
This line of research continued for a long time exploring relations between some classes of compact spaces and non-separable Banach spaces, for example Corson and Valdivia compact spaces, Weakly Lindel\"{o}f determined spaces (WLD) and Plichko spaces. This kind of spaces have been widely studied, we refer to \cite{Kalenda2} for a survey in these topics.\\
A compact space $K$ is called Corson if it is homeomorphic to a subset of
\begin{equation*}
\Sigma(\Gamma)=\{x\in\R^{\Gamma}: \mbox{ supp}(x) \mbox{ is countable} \}
\end{equation*}
for a set $\Gamma$. A compact space $K$ is called Valdivia if it is homeomorphic to some $K^{'}\subset\R^{\Gamma}$ with $K^{'}\cap\Sigma(\Gamma)$ dense in $K^{'}$.\\
In this work we will use retractional skeletons that yield a generalization of Valdivia and Corson compact spaces.
In \cite{KubisMicha} the authors introduced the definition of retractional skeleton and they proved that a compact space is Valdivia if and only if it has a commutative retractional skeleton. In \cite{Cuth1} it is proved that a compact space is Corson if and only if it has a full retractional skeleton. There is a dual formulation of retractional skeleton in Banach space, called projectional skeleton \cite{Kubis1}. This definition is strictly related with Plichko spaces and weakly Lindel\"{o}f determined spaces (WLD), mentioned above.
The paper is organized as follows.\\
In the remaining part of the introductory section notations and basic notions concerning topology and Banach space theory addressed in this paper are given.\\
In Section 2 the classes of non-commutative Corson countably compact spaces and weakly non-commutative Corson countably compact spaces are introduced. These notions are the non-commutative counterparts of similar notions introduced in \cite{Kalenda1}. Moreover several stability properties are studied.\\
In Section 3 the class of weakly non-commutative Valdivia compact spaces is introduced. Also in this case several stability properties are studied.\\
In Section 4 the definition of $[0,\eta)$-sum is recalled. Relations between $[0,\eta)$-sum and countably compact spaces are investigated. Results about ordinal spaces are given.\\
In Section 5 the definition of Aleksandrov duplicates is recalled. Some relations between Aleksandrov duplicates and retractional skeletons are given. \\
In Section 6 consequences of previous sections are studied in Banach space theory setting.\\
We denote with $\omega_{0}$ the set of natural numbers (including 0) with the usual order. Given a set $X$ we denote by $[X]^{\leq \omega_{0}}$ the family of all countable subsets of $X$ and by $|X|$ the cardinality of the set $X$. As usual we denote with $\aleph_{0}$ the smallest infinite cardinal.\\
All the topological spaces are assumed to be Hausdorff and completely regular. Given a topological space $T$ we denote by $\overline{A}$ the closure of $A\subset T$. We say that $A\subset T$ is countably closed if $\overline{C}\subset A$ for every $C\in [A]^{\leq \omega_{0}}$. A topological space $T$ is a Fr\'{e}chet-Urysohn space if for every $A\subset T$ and $x\in \overline{A}$ there is a sequence $\{x_{n}\}_{n\in\omega_{0}}\subset A$ such that $x_{n}\to x$. $\beta T$ denotes the \v{C}ech-Stone compactification of $T$. We use $S^1$ to indicate the complex numbers with absolute value equal to one. As in \cite{Cuth1}, we will use non-commutative Valdivia compacta to indicate the class of compact spaces with retractional skeleton.\\ 
Given a topological compact space $K$ we use $C(K)$ to indicate the space of all real-valued continuous function on $K$ or the space of all complex-valued continuous function on $K$ with the usual norm. Additionally we will use $C(K,\R)$ and $C(K,\C)$ where we want to differentiate. $P(K)$ stands for the space of probability measures with the weak$ ^* $-topology. If $\mu\in C(K)^*$, we denote by $|\mu|$ its total variation. If $\mu$ is a non-negative measure, we denote by supp$\mu$ the support of the measure $\mu$, i.e. the set of those points $x\in K$ such that each neighborhood of $x$ has positive $\mu$-measure. The support of a measure $\mu\in C(K)^*$ coincides with the support of its total variation $|\mu|$.\\
We shall consider Banach spaces over the field of real or complex numbers (most proofs work simultaneously for both cases, when necessary we will point out explicitly the differences). Given a Banach space $X$ and a subset $A\subset X$ we denote by span$(A)$ and conv$(A)$ the linear hull and the convex hull respectively. $B_{X}$ is the norm-closed unit ball of $X$ (i.e. the set $\{x\in X:\, \|x\|\leq 1\}$). As usual $X^{*}$ stands for the (topological) dual space of $X$. Given $A\subset X$ we denote by $A^{\perp}=\{x^{*}\in X^{*}:\, x^{*}(x)=0,\, \forall x\in A \}$. A set $D\subset X^{*}$ is said \textit{r-norming} if 
\begin{equation*}
\|x\|\leq r \sup\{|x^{*}(x)|:\,x^{*}\in D\cap B_{X^{*}}\}
\end{equation*}
for every $x\in X$. We say that a set $D\subset X^{*}$ is norming if it is \textit{r-norming} for some $r\geq 1$.\\

\section{Non-commutative Corson countably compact spaces}

In this paper we will use retractional skeletons also in countably compact setting. We recall the following definition.

\begin{defn}
A retractional skeleton in a countably compact space $X$ is a family of continuous retractions $\{r_{s}\}_{s\in\Gamma}$, indexed by an up-directed partially ordered set $\Gamma$, such that
\begin{enumerate}[$(i)$]
\item $r_{s}[X]$ is a metrizable compact space for each $s\in\Gamma$,
\item $s,t\in\Gamma$, $s\leq t$ then $r_{s}=r_t\circ r_s=r_s\circ r_t$,
\item given $s_{0}\leq s_{1}\leq...$ in $\Gamma$, $t=\sup_{n\in\omega_{0}}s_{n}$ exists and $r_{t}(x)=\lim_{n\to \infty}r_{s_{n}}(x)$ for every $x\in X$,
\item for every $x\in X$, $x=\lim_{s\in\Gamma}r_{s}(x)$.
\end{enumerate}
We say  that $D=\bigcup_{s\in\Gamma}r_{s}[X]$ is the set induced by the retractional skeleton $\{r_{s}\}_{s\in\Gamma}$ in $X$.\\
If $D=X$ we will say that $\{r_{s}\}_{s\in\Gamma}$ is a \textit{full retractional skeleton} and $X$ is a \textit{non-commmutative Corson countably compact space}.
\end{defn}

We recall some useful and well-known results about retractional skeletons.

\begin{thm}\cite[Theorem 32]{Kubis1}\label{PropRetr}
Assume $D$ is induced by a retractional skeleton in a compact space $K$. Then:
\begin{enumerate}[$(i)$]
\item $D$ is dense in $K$ and for every countable set $A\subset D$, $\overline{A}$ is metrizable and contained in $D$.
\item $D$ is a Fr\'{e}chet-Urysohn space;
\item $D$ is a normal space and $K=\beta D$.
\end{enumerate}
\end{thm}

In particular we observe that given a retractional skeleton, in a compact space $X$, its induced space $D$ is countably compact.

\begin{prop}\cite[Proposition 4.5]{CuthKalenda}\label{CaratCountably}
Let $X$ be a countably compact space. Then $X$ has a full retractional skeleton if and only if it is induced by a retractional skeleton in $\beta X$.\\
Moreover, if $\{r_{s}\}_{s\in\Gamma}$ is a full retractional skeleton in $X$, then there is a retractional skeleton $\{R_{s}\}_{s\in \Gamma}$ in $\beta X$ inducing $X$ such that $R_{s}\upharpoonright_X=r_{s}$ for every $s\in \Gamma$.
\end{prop}

\noindent
We observe that non-commutative Corson countably compact spaces are a generalization of Corson countably compact spaces given in \cite{Kalenda1}. Moreover, let $X$ be a countably compact space, it is a Corson countably compact space if and only if $X$ has a commutative full retractional skeleton. In fact:\\
$(\Rightarrow)$ Let $h: X\to [0,1]^{\Gamma}$ be a continuous injection of $X$ into $\Sigma(\Gamma)$, for some set $\Gamma$, then $\overline{h(X)}$ is a Valdivia compact space, hence by \cite[Theorem 6.1]{KubisMicha} it has a commutative retractional skeleton such that its induced subspace is $h(X)$. Hence $X$ has a commutative full retractional skeleton.\\
$(\Leftarrow)$ Suppose that $X$ has a commutative full retractional skeleton, then, by Proposition \ref{CaratCountably} $\beta X$ is a Valdivia compact space. By \cite[Theorem 6.1]{KubisMicha}, $X$ is a dense $\Sigma$-subspace of $\beta X$, hence it is a Corson countably compact space.

\begin{lem}\cite[Lemma 3.5]{Cuth1}\label{SubspaceRetra}
Let $K$ be a compact space, $F\subset K$ closed subset and let $D\subset K$ be such that $D$ is induced by a retractional skeleton in $K$. If $D\cap F$ is dense in $F$, then $D\cap F$ is induced by a retractional skeleton in $F$.
\end{lem}

\begin{prop}\cite[Proposition 31]{Kubis1}\label{ProdRetr}
The class of non-commutative Valdivia compacta is closed under arbitrary products. Moreover if $\{K_{n}\}_{n\in\omega_{0}}$ is a countable family of non-commutative Valdivia compact spaces and $D_{n}\subset K_{n}$ is an induced subspace for every $n\in\omega_{0}$, then $D=\prod\limits_{n\in\omega_{0}}D_{n}$ is an induced space of $K=\prod\limits_{n\in\omega_{0}}K_{n}$.
\end{prop}

Now we give the definition of weakly non-commutative Corson countably compact space, which is a generalization of weakly Corson countably compact space introduced in \cite{Kalenda1}.

\begin{defn}
Let $X$ be a countably compact space, we say that it is a \textit{weakly non-commutative Corson countably compact space} if there exists a continuous onto mapping $f:Y\to X$ such that $Y$ is a non-commutative Corson countably compact space.
\end{defn}

We now give the definition of the countably compact version of the one-point compactification.

\begin{defn}\label{def1pointmod}
Let $\{X_{\alpha}\}_{\alpha\in A}$ be a family of countably compact spaces, we say that $X=(\bigoplus\limits_{\alpha\in A}X_{\alpha})\cup\{\infty\}$ is an \textit{one-point countably compact modification of topological sum} if $X$ is countably compact and each $X_{\alpha}$ is a clopen subset of $X$.
\end{defn}

We observe that the previous definition is different from the definition of one-point modification given in \cite{Kalenda1}. Using that definition, Lemma 2.1 of \cite{Kalenda1} is not correct. In fact, let $A=\{1,2\}$, $X_{1}=X_{2}=[0,\omega_{1})$ with usual topology and $X$ be the one-point compatification of $X_{1}\oplus X_{2}$, moreover we observe that by \cite[Example 1.10]{Kalenda2} the space $X$ is not Valdivia. Let $\Gamma_{i}=[0,\omega_{1})$ and
\begin{equation*}
\begin{split}
f_{i}:&X_{i}\to \Sigma(\Gamma_{i})\\
&\alpha\mapsto\chi_{[0,\alpha)},
\end{split}
\end{equation*}
for $i=1,2$. As in \cite[Lemma 2.1]{Kalenda1} we define $\Gamma=\{(i,\gamma):i\in A,\gamma\in\Gamma_{\alpha}\}\cup \{(i,A):i\in A\}$ and $f:X\to \R^{\Gamma}$ 
\begin{equation*}
f(x)(i,\gamma)=\begin{cases}
f_{i}(x)(\gamma) &x\in X_{i},\gamma\in\Gamma_{i}\\
1,&x\in X_{i},\gamma=A\\
0, &\mbox{otherwise.}
\end{cases}
\end{equation*}
\underline{Claim:} $f$ is not continuous. In fact, let $\{\alpha\}_{\alpha<\omega_{1}}\subset X_{1}$, it is a converging net to $\infty$ in $X$. Using the definition of $f$, $\{f(\alpha)\}_{\alpha<\omega_{1}}$ does not converge to $f(\infty)$ in $\R^{\Gamma}$. Hence it cannot be continuous.\\
Finally, we observe that using Definition \ref{def1pointmod}, Lemma 2.1 of \cite{Kalenda1} is correct.

\begin{lem}\label{Operznoncomm}
The class of non-commutative Corson countably compact spaces is closed under
\begin{enumerate}[(1)]
\item countably closed subspaces,
\item countable products,
\item finite topological sums,
\item one-point countably compact modifications of topological sums,
\item quotient images.
\end{enumerate}
\end{lem}

\begin{proof}
\begin{enumerate}[$(1)$]
\item Let $Y$ be a countably closed subspace of a non-commutative Corson countably compact space $X$ then it is a countably compact space. Using Proposition \ref{CaratCountably} and Lemma \ref{SubspaceRetra} it follows that $Y$ is a non-commutative Corson countably compact space.

\item Let $\{X_{n}\}_{n\in\omega_{0}}$ be a countable family of non-commutative Corson countably compact spaces. For every $n\in\omega_0$, by Proposition \ref{CaratCountably}, $X_n$ is an induced subspace of the non-commutative Valdivia compact space $\beta X_n$. By Proposition \ref{ProdRetr} $\prod_{n\in\omega_0}\beta X_n$ is non-commutative Valdivia and $\prod_{n\in\omega_0} X_n $ is an induced subspace. Hence $\prod_{n\in\omega_0} X_n $ is a non-commutative Corson countably compact space.

\item Let $X_{1},...,X_{n}$ be a finite collection of non-commutative Corson countably compact spaces and define the topological sum $X=\bigoplus\limits_{k=1}^{n}X_{k}$. Using the countably compactness of every $X_{k}$ it is easy to prove that $X$ is countably compact.\\
It remains to prove that $X$ has full retractional skeleton. For every $k=1,...,n$ $X_{k}$ has a full retractional skeleton, then let $(\Gamma_{k},\leq_{k})$ be an up-directed partially ordered set and $\{r_{s}^{k}\}_{s\in\Gamma_{k}}$ be a full retractional skeleton on $X_{k}$. Now we define a family of retractions on $X$, let
\begin{equation*}
\Gamma=\{\gamma=(\gamma_{1},...,\gamma_{n})\in \Gamma_{1}\times...\times\Gamma_{n}\}
\end{equation*}
equipped with the following order: given $\gamma,\delta\in\Gamma$ we will say that $\gamma\leq \delta$ if and only if $\gamma_{k}\leq_{k} \delta_{k}$ for every $k=1,...,n$.\\
For every $\gamma \in \Gamma$ we define $r_{\gamma}:X\to X$ as follows: given $x\in X$ we have $x\in X_{k}$ for some $k$, then we put
\begin{equation*}
r_{\gamma}(x)=r_{\gamma_{k}}^{k}(x).
\end{equation*} 
Since $r_{\gamma}$ is continuous on every $X_{k}$, it is continuous on $X$. Moreover, since $\{r_{s}^{k}\}_{s\in\Gamma_{k}}$ is a full retractional skeleton on $X_k$ for every $k=1,...,n$, it is easy to check that $\{r_{\gamma}\}_{\gamma\in\Gamma}$ is a full retractional skeleton on $X$.

\item Let $\{X_{\alpha}\}_{\alpha\in A}$ be a family of non-commutative Corson countably compact spaces and $X=(\bigoplus\limits_{\alpha\in A}X_{\alpha})\cup\{\infty\}$ be a one-point countably compact modification of topological sum of them.\\
For every $\alpha\in A$ there exist an up-directed partially ordered set $(\Gamma_{\alpha},\preceq_{\alpha})$ and a full retractional skeleton $\{r_{s}^{\alpha}\}_{s\in\Gamma_{\alpha}}$ on $X_{\alpha}$.\\
We define $\Gamma^{'}_{\alpha}=\Gamma_{\alpha} \cup \{0\}$ and a relation $\leq_{\alpha}$ such that if we restrict $\leq_{\alpha}$ to $\Gamma_{\alpha}\times \Gamma_{\alpha}$ we have the same order of $(\Gamma_{\alpha},\preceq_{\alpha})$ and $0\leq_{\alpha} s$, for every $s\in \Gamma_{\alpha}$. This way $(\Gamma_{\alpha}^{'},\leq_{\alpha})$ is an up-directed partially ordered set, for every $\alpha\in A$.\\
Let
\begin{equation*}
\Gamma=\{\gamma\in \prod_{\alpha\in A}\Gamma_{\alpha}^{'}: |S(\gamma)|\leq\aleph_{0}  \}
\end{equation*}
where $S(\gamma)=\{\alpha\in A:\gamma(\alpha)\neq 0 \}$. Given $\gamma_{1},\gamma_{2}\in \Gamma$ we will say that $\gamma_{1}\leq \gamma_{2}$ if and only if $\gamma_{1}(\alpha)\leq_{\alpha} \gamma_{2}(\alpha)$ for every $\alpha\in A$. For every $\gamma \in \Gamma $ we define the retraction $r_{\gamma}:X\to X$ as follows:
\begin{itemize}
\item if $x=\infty$, $r_{\gamma}(x)=\infty$,
\item if $x\neq \infty$ there exists an $\alpha \in A$ such that $x\in X_{\alpha}$ then 
\begin{equation*}
r_{\gamma}(x)=\begin{cases}
r_{\gamma(\alpha)}^{\alpha}(x) &\mbox{ if } \gamma(\alpha)\neq 0,\\
\infty &\mbox{ if } \gamma(\alpha)= 0,
\end{cases}
\end{equation*}
\end{itemize}
\underline{Claim}: for every $\gamma\in \Gamma$ the retraction $r_{\gamma}$ is a continuous mapping.\\
Let $\gamma \in\Gamma$ we study the continuity of $r_{\gamma}$ at each point:
\begin{itemize}
\item if $x=\infty$, let $U$ be an open neighborhood of $x$ then $X\setminus U$ is a countably compact space. Moreover, since each $X_{\alpha}$ is open and $X\setminus U$ is a countably compact space, we have that the cardinality of $F=\{\alpha\in A: (X\setminus U)\cap X_{\alpha}\neq\emptyset\}$ is finite. Let $V=X\setminus \bigcup_{\alpha\in F}X_{\alpha}$, it is a neighborhood of $x$ and $V\subset U$. By definition of $r_{\gamma}$ we have $r_{\gamma}(V)\subset V$; hence we have the continuity in $x$.
\item if $x\neq \infty$, since each $X_{\alpha}$ is clopen and the restriction of $r_{\gamma}$ on each $X_{\alpha}$ is continuous, we conclude that $r_{\gamma}$ is continuous in $x$.
\end{itemize}
It proves the claim.\\
It remains to prove that $\{r_{\gamma}\}_{\gamma\in\Gamma}$ is a full retractional skeleton on $X$.
\begin{enumerate}[$(i)$]
\item Since $r_{\gamma}[X]=(\bigoplus\limits_{\alpha\in S(\gamma)}r_{\gamma(\alpha)}^{\alpha}[X_{\alpha}])\cup\{\infty\}$ is countably compact, regular and has a countable network we have that it is metrizable and compact.
\item Let $\gamma_{1}, \gamma_{2}\in \Gamma$ such that $\gamma_{1}\leq\gamma_{2}$. If $x=\infty$ it is trivial that  $r_{\gamma_{1}}(x)=r_{\gamma_{1}}\circ r_{\gamma_{2}}(x)=r_{\gamma_{2}}\circ r_{\gamma_{1}}(x)$. Then, suppose $x\in X_{\alpha}$ for some $\alpha\in A$, three cases are possible:
\begin{itemize}
\item $\gamma_{1}(\alpha)= 0$ and $\gamma_{2}(\alpha)= 0$ we have
\begin{equation*}
r_{\gamma_{1}}(x)=r_{\gamma_{1}}\circ r_{\gamma_{2}}(x)=r_{\gamma_{2}}\circ r_{\gamma_{1}}(x)=\infty.
\end{equation*}
\item $\gamma_{1}(\alpha)= 0$ and $\gamma_{2}(\alpha)\neq 0$ we have
\begin{equation*}
r_{\gamma_{1}}( r_{\gamma_{2}}(x))=r_{\gamma_{1}}( r_{\gamma_{2}(\alpha)}^{\alpha}(x))=\infty=r_{\gamma_{1}}(x)
\end{equation*}
and
\begin{equation*}
r_{\gamma_{2}}( r_{\gamma_{1}}(x))=r_{\gamma_{2}}( \infty)=\infty=r_{\gamma_{1}}(x)
\end{equation*}
\item $\gamma_{1}(\alpha)\neq 0$ and $\gamma_{2}(\alpha)\neq 0$ we have:
\begin{equation*}
r_{\gamma_{1}}( r_{\gamma_{2}}(x))=r_{\gamma_{1}(\alpha)}^{\alpha}( r_{\gamma_{2}(\alpha)}^{\alpha}(x))=r_{\gamma_{1}(\alpha)}^{\alpha}(x)=r_{\gamma_{1}}(x)
\end{equation*}
and
\begin{equation*}
r_{\gamma_{2}}( r_{\gamma_{1}}(x))=r_{\gamma_{2}(\alpha)}^{\alpha}( r_{\gamma_{1}(\alpha)}^{\alpha}(x))=r_{\gamma_{1}(\alpha)}^{\alpha}(x)=r_{\gamma_{1}}(x).
\end{equation*}
\end{itemize}
\item Let $\gamma_{1}\leq \gamma_{2}\leq ...$ and $\gamma=\sup_{n\in\omega_{0}}\gamma_{n}$, then for every $\alpha \in A$ we have $\gamma(\alpha)=\sup_{n\in\omega_{0}}\gamma_{n}(\alpha)$. Let $x\in X$, two cases are possible
\begin{itemize}
\item $x=\infty$: we have $r_{\gamma_{n}}(x)=\infty$ for every $n\in\omega_{0}$ then $\lim\limits_{n\in\omega_{0}}r_{\gamma_{n}}(x)=x$.
\item $x\neq \infty$: there exists $\alpha\in A$ such that $x\in X_{\alpha}$. If $\alpha\notin S(\gamma)$ then $\alpha\notin S(\gamma_{n})$ for each $n\in\omega_{0}$; hence $r_{\gamma}(x)=r_{\gamma_{n}}(x)=\infty$. If $\alpha\in S(\gamma)$, then there exists $n_{0}\in\omega_{0}$ such that for every $n\geq n_{0}$ we have $\alpha\in S(\gamma_{n})$; hence using the fact that the family $\{r_{s}^{\alpha}\}_{s\in\Gamma_{\alpha}}$ is a full retractional skeleton on $X_{\alpha}$ we have that $r_{\gamma}(x)=r_{\gamma(\alpha)}^{\alpha}(x)=\lim\limits_{n\in\omega_{0}}r_{\gamma_{n}(\alpha)}^{\alpha}(x)=\lim\limits_{n\in\omega_{0}}r_{\gamma_{n}}(x)$.
\end{itemize}
\item For every $\alpha \in A$ and $x\in X_{\alpha}$ there exists $s\in \Gamma_{\alpha}$ such that $r_{s}^{\alpha}(x)=x$ then for every $\gamma\in \Gamma$ such that $s\leq_{\alpha}\gamma(\alpha)$ we have $r_{\gamma}(x)=x$. Therefore $\lim_{\gamma\in\Gamma}r_{\gamma}(x)=x$. The case $x=\infty$ is trivial.
\end{enumerate}
Finally it is trivial that $X=\bigcup_{\gamma\in\Gamma}r_{\gamma}[X]$.
\item It follows immediately combining \cite[Theorem 1.1]{CuthKalenda} and \cite[Theorem 3.6]{RojasTkachuk}.
\end{enumerate}
This completes the proof.
\end{proof}

Using the same argument as in \cite[Lemma 2.2]{Kalenda1} it is possible to prove the following result about stability properties of weakly non-commutative Corson countably compact spaces.

\begin{lem}\label{OperzWnoncomm}
The class of weakly non-commutative Corson countably compact is closed under
\begin{enumerate}[(1)]
\item countably closed subspaces,
\item countable products,
\item continuous images,
\item finite unions,
\item finite topological sums,
\item one-point countably compact modifications of topological sums.
\end{enumerate}
\end{lem}

\section{Weakly non-commutative Valdivia compact spaces}

Now we give the definition of weakly non-commutative Valdivia compact space which is a generalization of the commutative one introduced in $\cite{Kalenda1}$.

\begin{defn}
A compact space $K$ is said \textit{weakly non-commutative Valdivia compact} if it has a dense countably compact subspace which is weakly non-commutative Corson.
\end{defn}

Next two results are the non-commutative version of \cite[Proposition 3.1]{Kalenda1} and \cite[Lemma 1.17]{Kalenda2}.

\begin{prop}
A compact space $K$ is weakly non-commutative Valdivia if and only if it is a continuous image of a non-commutative Valdivia compact.
\end{prop}
\begin{proof}
We start by proving the \textquotedblleft if part\textquotedblright. Let $L$ be a non-commutative Valdivia compact space and $f:L\to K$ be a continuous onto mapping.\\
Let $D$ be an induced subspace of $L$, hence by Theorem \ref{PropRetr} it is a dense non-commutative Corson countably compact space. Since $f$ is a continuous mapping we have that $f(D)\subset K$ is a dense weakly non-commutative Corson countably compact space. Hence $K$ is a weakly non-commutative Valdivia compact space.\\
Conversely let $D$ be a dense weakly non-commutative Corson countably compact subspace of $K$. Then there exist a non-commutative Corson countably compact space $A$ and a continuous surjection $f:A\to D$. By Proposition \ref{CaratCountably} we have that $\beta A$ is a non-commutative Valdivia compact space. Let $\beta f:\beta A\to K$ be the continuous extension of $f$.\\
Since $D$ is dense in $K$, $D\subset \beta f(\beta A)$ and $\beta f(\beta A)$ is closed we have that $\beta f$ is a surjection. Thus $K$ is a continuous onto image of a non-commutative Valdivia compact space.
\end{proof}

We prefer to omit the proof of the following result because it is completely analogous to \cite[Lemma 1.17]{Kalenda2}.

\begin{prop}\label{compattinonwncV}
Let $K$ be a compact space with a countable dense set of $G_{\delta}$ points. If $K$ is a continuous image of a non-commutative Valdivia compact space then $K$ is metrizable.
\end{prop}

\begin{cor}\label{compattinonwncCcc}
Let $K$ be a compact space with a countable dense set of $G_{\delta}$ points. If $K$ is a continuous image of a non-commutative Corson countably compact space then $K$ is metrizable.
\end{cor}

Using Proposition \ref{compattinonwncV}, Corollary \ref{compattinonwncCcc} and \cite[Example 1.18]{Kalenda2} we have some examples of compact spaces which are neither weakly non-commutative Corson nor weakly non-commutative Valdivia.

\begin{prop}\label{1pointcpt}
Let $\{K_{\alpha}\}_{\alpha\in A}$ be a family of non-commutative Valdivia compact spaces. Then the one-point compactification of $K=\bigoplus\limits_{\alpha\in A}K_{\alpha}$ is a non-commutative Valdivia compact space.
\end{prop}

We will not provide the full proof because we use the same idea of Lemma \ref{Operznoncomm}, point (4).

\begin{proof}
We will use the same notations of point $(4)$ Lemma \ref{Operznoncomm}. We define the up-directed partially ordered set $\Gamma$ and the family of retractions $\{r_{\gamma}\}_{\gamma\in\Gamma}$ as well. Moreover, we observe that each $K_{\alpha}$ is clopen in the one-point compactification of $K$, hence the continuity of every $r_{\gamma}$ follows in the same way of point $(4)$ Lemma \ref{Operznoncomm}.\\
It remains to prove that $\{r_{\gamma}\}_{\gamma\in\Gamma}$ is a retractional skeleton. Points $(i),(ii)$ and $(iii)$ follow as in point $(4)$ Lemma \ref{Operznoncomm}.\\
$(iv)$ If $x=\infty$, $r_{\gamma}(x)=\infty$ for every $\gamma \in\Gamma$, hence it is clear that $\lim\limits_{\gamma\in\Gamma}r_{\gamma}(x)=x$. Suppose otherwise $x\in K_{\alpha}$ for some $\alpha\in A$, by definition of $(\Gamma,\leq)$ there exists $\gamma_{0}\in\Gamma$ such that $\gamma(\alpha)\neq 0$ for every $\gamma\geq \gamma_{0}$, hence we have $r_{\gamma}(x)=r_{\gamma(\alpha)}^{\alpha}(x)$ for such $\gamma$. Since $\{r_{s}^{\alpha}\}_{s\in\Gamma_{\alpha}}$ is a retractional skeleton on $K_{\alpha}$, we deduce that $\lim\limits_{\gamma\in\Gamma}r_{\gamma}(x)=\lim\limits_{s\in\Gamma_{\alpha}}r_{s}^{\alpha}(x)=x$. 
\end{proof}

\begin{prop}\label{operazWNVald}
The class of weakly non-commutative Valdivia compact spaces is closed under
\begin{enumerate}[(1)]
\item arbitrary products,
\item one point compactifications of arbitrary topological sums,
\item continuous images,
\item finite unions.
\end{enumerate}
Moreover if $K$ is weakly non-commutative Valdivia compact and $L$ is a subset of $K$ which can be written as the closure of the union of an arbitrary family of $G_{\delta}$ subsets of $K$, then $L$ is weakly non-commutative Valdivia as well.
\end{prop}
\begin{proof}
\begin{enumerate}[(1)]
\item Let $\{K_{\alpha}\}_{\alpha\in A}$ be a family of weakly non-commutative Valdivia compact spaces. Then there exists a continuous onto mapping $f_{\alpha}:L_{\alpha}\to K_{\alpha}$, where $L_{\alpha}$ is a non-commutative Valdivia compact space for every $\alpha\in A$.\\
We define $K=\prod\limits_{\alpha\in A}K_{\alpha}$ and $L=\prod\limits_{\alpha\in A}L_{\alpha}$. By Proposition \ref{ProdRetr} $L$ is a non-commutative Valdivia compact space. Finally let $f:L\to K$ defined by $f(y)(\alpha)=f_{\alpha}(y(\alpha))$, it is clearly onto and moreover, since it is coordinatewise continuous, it is continuous.
\item It follows, using Proposition \ref{1pointcpt}, by the same argument of the previous point.\\
\end{enumerate}
\noindent
Points $(3)$-$(4)$ are trivial.
Finally let $K$ be a weakly non-commutative Valdivia compact and $L\subseteq K$ such that $L=\overline{\bigcup_{\beta \in B}Z_{\beta}}$, where $B$ is a set and $Z_{\beta}$ is a $G_{\delta}$ subset of $K$, for every $\beta\in B$. Let $D\subset K$ be a dense weakly non-commutative Corson countably compact subspace. Now we want to prove that $L\cap D$ is dense in $L$. In fact 
\begin{equation*}
L\cap D=\overline{\bigcup_{\beta \in B}Z_{\beta}}\cap D\supset (\bigcup_{\beta \in B}Z_{\beta})\cap D= \bigcup_{\beta \in B}(Z_{\beta}\cap D).
\end{equation*} 
Taking the closure
\begin{equation*}
\overline{L\cap D}\supset \overline{\bigcup_{\beta \in B}(Z_{\beta}\cap D)}\supset \bigcup_{\beta \in B}\overline{(Z_{\beta}\cap D)}\supset\bigcup_{\beta \in B}Z_{\beta}.
\end{equation*} 
In the last part we have used \cite[Lemma 3.3]{Cuth1}. Therefore we have $L\supset \overline{L\cap D}\supset L$.
Moreover $L\cap D$ is a closed subspace of $D$, hence it is weakly non-commutative Corson, therefore $L$ is weakly non-commutative Valdivia.
\end{proof}

Now we want to use the previous results to prove the equivalent of \cite[Theorem 3.6]{Kalenda1} in the non-commutative case.

\begin{thm}\label{TeoremaMisurediradon}
Let $K$ be a compact space. Consider the following assertions.
\begin{enumerate}[(1)]
\item $K$ is weakly non-commutative Valdivia.
\item $(B_{C(K,\C)^{*}},w^{*})$ is weakly non-commutative Valdivia.
\item $(B_{C(K,\R)^{*}},w^{*})$ is weakly non-commutative Valdivia.
\item $P(K)$ is weakly non-commutative Valdivia.
\end{enumerate}
Then $1\Rightarrow 2\Leftrightarrow 3 \Leftrightarrow 4$. If $K$ has a dense set of $G_{\delta}$ points, then all three assertions are equivalent.
\end{thm}
\begin{proof}
$(1)\Rightarrow(2)$ Let $K$ be a weakly non-commutative Valdivia compact space, then it is a continuous image of a non-commutative Valdivia compact $L$. Using \cite[Proposition 28]{Kubis1} $C(L,\C)$ has 1-projectional skeleton then it is clear that $(B_{C(L,\C)^*},w^*)$ is non-commutative Valdivia. Hence, since $(B_{C(K,\C)^*},w^*)$ is a continuous image of $(B_{C(L,\C)^*},w^*)$, it is a weakly non-commutative Valdivia compact space.\\
$(2)\Rightarrow(3)$ Suppose that $(B_{C(K,\C)^*},w^*)$ is a weakly non-commutative Valdivia compact space.\\
We want to show that $(B_{C(K,\R)^*},w^*)$ is a continuous images of $(B_{C(K,\C)^*},w^*)$. To do that, consider the following map:
\begin{equation*}
\begin{split}
\varphi:(B_{C(K,\C)^*},w&^*) \to (B_{C(K,\R)^*},w^*)\\
&\mu\mapsto \mbox{Re}(\mu).
\end{split}
\end{equation*}
It is clear that it is a surjection. To prove that $\varphi$ is a continuous mapping, it is sufficient to observe that for every $f\in C(K,\R)$ we have Re$(\mu)(f)=$Re$\mu(f)$. \\
$(4)\Rightarrow(2)$ Suppose $P(K)$ is weakly non-commutative Valdivia. By Proposition \ref{operazWNVald} the space $P(K)\times S^1$ is weakly non-commutative Valdivia, finally by $(1)\Rightarrow(4)$ $P(P(K)\times S^1)$ is weakly non-commutative Valdivia. By \cite[Proposition 2.38]{LukesMalyNetSpu} the barycenter mapping 
\begin{equation*}
r:P(B_{C(K,\C)^*})\to (B_{C(K,\C)^{*}},w^*)
\end{equation*}
is surjective and continuous. Moreover, since $P(K)\times S^1$ contains all the extreme points of $(B_{C(K,\C)^{*}},w^*)$, using \cite[Theorem 2.31]{LukesMalyNetSpu} we obtain that the restriction of $r$ to $P(P(K)\times S^1)$ is surjective as well. Hence, since $P(P(K)\times S^1)$ is weakly non-commutative Valdivia, we have that $(B_{C(K,\C)^{*}},w^*)$ is weakly non-commutative Valdivia, too.\\
$(4)\Rightarrow(1)$ If $K$ has a dense set of $G_{\delta}$ points, it follows in the same way as \cite[Theorem 3.6]{Kalenda1}.
\end{proof}

Now we give the non-commutative version of \cite[Theorem 3.7]{Kalenda1}, we recall the definition of property $(M)$.

\begin{defn}
A Compact space $K$ is said to have the \textit{property $(M)$} if every Radon probability measure on $K$ has separable support.
\end{defn}
\begin{thm}\label{TeoPropM}
Let $K$ be a weakly non-commutative Corson compact space with property $(M)$, then $(B_{C(K)^{*}},w^{*})$ is weakly non-commutative Corson as well.
\end{thm}
\begin{proof}
The proof of the real case follows as in the commutative case \cite[Theorem 3.7]{Kalenda1}, using Proposition \ref{supportoNumerabile} below instead of \cite[Proposition 5.1]{Kalenda2}.\\
Suppose that $K$ is a weakly non-commutative Corson compact space with property $(M)$, then using the real case and Lemma \ref{OperzWnoncomm},  $(B_{C(K,\R)^{*}},w^{*})\times (B_{C(K,\R)^{*}},w^{*})$ is weakly non-commutative Corson as well. Now consider
\begin{equation*}
\psi: (B_{C(K,\R)^{*}},w^{*})\times (B_{C(K,\R)^{*}},w^{*})\to (C(K,\C)^{*},w^{*}),
\end{equation*}
defined by $\psi(\mu,\nu)=\mu+i \nu$. $\psi$ is clearly continuous, hence $\psi((B_{C(K,\R)^{*}},w^{*})\times (B_{C(K,\R)^{*}},w^{*}))$ is weakly non-commutative Corson. Finally, since $(B_{C(K,\C)^{*}},w^{*})$ is a weak$^*$ compact space and $(B_{C(K,\C)^{*}},w^{*})\subset \psi((B_{C(K,\R)^{*}},w^{*})\times (B_{C(K,\R)^{*}},w^{*})) $, it is a weak$^*$ closed subspace of $\psi((B_{C(K,\R)^{*}},w^{*})\times (B_{C(K,\R)^{*}},w^{*}))$. Therefore, by Lemma \ref{OperzWnoncomm}, $(B_{C(K,\C)^{*}},w^{*})$ is a weakly non-commutative Corson compact space.
\end{proof}

Now we give the non-commutative version of \cite[Proposition 5.1]{Kalenda2}. To sake of completeness we will give the full proof although the last part is the same of the commutative one.
\begin{prop}\label{supportoNumerabile}
Let $K$ be a non-commutative Valdivia compact and $D\subset K$ be an induced subspace. Then the set
\begin{equation*}
S=\{\mu\in C(K)^{*}: \mbox{ supp}(\mu) \mbox{ is a separable subset of } D\}
\end{equation*}
is $1$-norming and induced by a 1-Projectional skeleton in $C(K)$.
\end{prop}
\begin{proof}
The real case follows by combining the first part of the complex case below and the second part of \cite[Proposition 5.1]{Kalenda2}.\\
Let $\{r_{s}\}_{s\in\Gamma}$ be a retractional skeleton on $K$ such that $D=\bigcup_{s\in\Gamma}r_{s}[K]$. It is standard to define a 1-Projectional skeleton $\{P_{s}\}_{s\in\Gamma}$ in $C(K,\C)$ as follow
\begin{equation*}
P_{s}(f)=f\circ r_{s}.
\end{equation*}
Let $S=\bigcup_{s\in\Gamma}P^{*}_{s}(C(K,\C)^{*})$ be the induced subspace. It is well known that it is 1-norming (hence weak$^{*}$-dense in $C(K,\C)^{*}$) linear and weak$^{*}$-countably closed. Moreover, $(B_{C(K,\C)^*},w^*)$ has retractional skeleton and $S\cap(B_{C(K,\C)^*},w^*) $ is an induced subspace.\\
Now we want to prove that
\begin{equation*}
S=\{\mu \in C(K,\C)^*: \mbox{ supp}\mu \mbox{ is a separable subset of } D\}.
\end{equation*}
We will prove the double inclusion.
\begin{itemize}
\item \textquotedblleft$\supseteq$\textquotedblright $\:$ Let $\mu$ be a real measure in the set on the right-hand side, then using the same argument of \cite[Proposition 5.1]{Kalenda2} we obtain $\mu\in S$.\\
Now, let $\mu$ be a complex measure in the set on the right-hand side, then, by the previous sentence, its total variation $|\mu|$ belongs to $S$ . Hence there exists $s_0\in\Gamma$ such that $P_{s}^{*}|\mu|=|\mu|$ for every $s\geq s_{0}$, in particular, by Riesz representation theorem, for every $f\in C(K,\C)$ we have
\begin{equation}\label{supportoNumerabile1}
\int_{K}fd|\mu|=\int_{K}f\circ r_{s} d|\mu|.
\end{equation}
Moreover, by the Radon-Nikod\'{y}m theorem there exists a measurable function $h$ such that $d\mu=h \,d|\mu|$ and $|h(x)|=1$ for every $x\in K$.\\
\underline{Claim:} There exists a $t\in\Gamma$ such that for every continuous function $f$ the equality
\begin{equation*}
\int_{K}f\cdot h\, d|\mu|=\int_{K}(f\circ r_{t})\cdot h\, d|\mu|.
\end{equation*}
holds.\\
Indeed, let $t\in \Gamma$ such that $t\geq s_0$ and supp$\mu\subset r_{t}[K]$: such $t$ exists by the $\sigma$-completeness of $\Gamma$ and the separability of supp$\mu$. Finally, let $f\in C(K,\C)$ and $\varepsilon>0$ then by the density of continuous function in $L^1(|\mu|)$ there exists $g\in C(K,\C)$ such that $\int_K |f\cdot h-g|\,d|\mu|<\varepsilon$; then using (\ref{supportoNumerabile1}) and the fact that supp$\mu\subset r_{t}[K]$, we obtain
\begin{equation*}
\begin{split}
\left|\int_{K}f\cdot h-(f\circ r_{t})\cdot h\, d|\mu|\right|&\leq \int_K |f\cdot h-g|\,d|\mu| + \int_K |(f\circ r_{t})\cdot h-g\circ r_t|\,d|\mu|\\
&<\varepsilon + \int_{r_t[K]} |(f\circ r_{t})\cdot h-g\circ r_t|\,d|\mu|\\
&=\varepsilon + \int_{r_t[K]} |f\cdot h-g|\,d|\mu|<2\varepsilon.
\end{split}
\end{equation*}
It proves the claim. Therefore $P^{*}_{t}\mu=\mu$, hence $\mu\in S$.

\item \textquotedblleft$\subseteq$\textquotedblright $\:$ Let $S^{'}=\mbox{span}\{\delta_{x}:x\in D\}$, since $S$ is linear and $\delta_{x}\in S$ for every $x\in D$ we have $S^{'}\subset S$. Moreover since $D$ is dense in $K$ we have that $S^{'}$ is $1$-norming. Then $S^{'}\cap B_{C(K,\C)^{*}}$ is weak$^*$ dense in $B_{C(K,\C)^{*}}$. In particular, every $\mu\in S\cap  B_{C(K,\C)^{*}}$ belongs to the weak$^{*}$ closure of $S^{'}\cap B_{C(K,\C)^{*}}$. Hence, since $S\cap  B_{C(K,\C)^{*}}$ is a weak$^*$ Fr\'{e}chet-Urysohn space, there exists a sequence $\{\mu_{n}\}_{n\in\omega_{0}}\subset S^{'}\cap B_{C(K,\C)^{*}}$ such that $\mu_{n}\overset{w^{*}}{\to} \mu$.\\
Let 
\begin{equation*}
C=\{x\in K: \: \exists\, n\in\omega_{0} \:\mu_{n}(\{x\})\neq 0\},
\end{equation*}
clearly $|C|\leq\aleph_{0}$. Moreover $\mu$ is supported by $\overline{C}$. Therefore, since $\overline{C}\subset D$ is separable, it is metrizable. Since supp$\mu\subset\overline{C} $ and $\overline{C}$ is metrizable and separable, we have that supp$\mu$ is separable as well.
\end{itemize}
\end{proof}

\section{$[0,\eta)$-sums}

We recall the definition of $[0,\eta)$-sum, introduced in \cite{Kalenda1}. Given an ordinal $\eta$ we will denote with $I(\eta)$ the subset of all isolated ordinals less than $\eta$. Let $\{X_{\alpha}\}_{\alpha\in I(\eta)}$ be a family of topological spaces, the $[0,\eta)$-sum is the set
\begin{equation*}
X=\{(\alpha,x):x\in X_{\alpha}, \alpha<\eta \mbox{ isolated} \} \cup \{(\alpha,\alpha): \alpha<\eta \mbox{ limit}\}
\end{equation*}
equipped with the following topology. Whenever $\alpha$ is isolated, the set $\{\alpha\}\times X_{\alpha}$ is canonically homeomorphic to $X_{\alpha}$ and clopen in $X$. A neighborhoods basis for $(\alpha,\alpha)$ if $\alpha$ is limit is formed by sets
\begin{equation*}
B_{\gamma}((\alpha,\alpha))=\{(\beta,x)\in X: \gamma<\beta<\alpha \}.
\end{equation*}
Since in our setting $X_{\alpha}$ is Hausdorff and completely regular, $X$ is Hausdorff and completely regular too.\\
Let $\eta$ be an uncountable cofinality ordinal and $X_{\alpha}$ be a countably compact space, for every $\alpha\in I(\eta)$. Let $Y$ be the $[0,\eta)$-sum of $\{X_{\alpha}\}_{\alpha\in I(\eta)}$ and $X$ be the topological subspace of $Y$ defined by
\begin{equation*}
\begin{split}
X=&\{(\alpha,x): x\in X_{\alpha}, \: \alpha<\eta \mbox{ isolated} \}\cup\\
&\{(\alpha,\alpha):\alpha<\eta \mbox{ limit with countable cofinality} \}.
\end{split}
\end{equation*}
We will say that $X$ is the \textit{countably $[0,\eta)$-sum} of $\{X_{\alpha}\}_{\alpha\in I(\eta)}$.
We observe that the countably $[0,\eta)$-sum is a countably closed subset of the $[0,\eta)$-sum. Then, since, by \cite[Lemma 2.3]{Kalenda1}, the $[0,\eta)$-sum is countably compact, we conclude that the countably $[0,\eta)$-sum is countably compact as well.
\begin{lem}\label{etaSomme}
Let $\eta$ be an uncountable cofinality ordinal and $\{X_{\alpha}\}_{\alpha\in I(\eta)}$ be a family of non-commutative Corson countably compact spaces. Let $X$ be the countably $[0,\eta)$-sum of $\{X_{\alpha}\}_{\alpha\in I(\eta)}$. Then $X$ is a non-commutative Corson countably compact space.
\end{lem}

\begin{proof}
For every isolated $\alpha<\eta$ there exist an up-directed partially ordered set $(\Gamma_{\alpha},\preceq_{\alpha})$ and a full retractional skeleton $\{r_{s}^{\alpha}\}_{s\in\Gamma_{\alpha}}$ on $X_{\alpha}$.\\ 
We define $\Gamma^{'}_{\alpha}=\Gamma_{\alpha} \cup \{0\}$ and a relation $\leq_{\alpha}$ such that if we restrict $\leq_{\alpha}$ to $\Gamma_{\alpha}\times \Gamma_{\alpha}$ we have the same order of $(\Gamma_{\alpha},\preceq_{\alpha})$ and $0\leq_{\alpha} s$, for every $s\in \Gamma_{\alpha}$. This way $(\Gamma_{\alpha}^{'},\leq_{\alpha})$ is an up-directed partially ordered set.\\
For every $\gamma\in \prod_{\alpha\in I(\eta)}\Gamma_{\alpha}^{'}$ let us define $S(\gamma)=\{\alpha\in I(\eta): \gamma(\alpha)\neq 0\}$. Now we are going to define the partially ordered set and the family of retractions on $X$. In order to do this, let
\begin{equation*}
\begin{split}
\Gamma=&\{\gamma\in\prod_{\alpha\in I(\eta)}\Gamma_{\alpha}^{'}: |S(\gamma)|\leq \aleph_{0}, \, 0\in S(\gamma), \mbox{ and if } \alpha\in S(\gamma)\\ &\mbox{ then } \beta \in S(\gamma) \mbox{ whenever } \alpha<\beta<\alpha+\omega_{0}  \}.
\end{split}
\end{equation*}
Given $\gamma_{1},\gamma_{2}\in \Gamma$ we will say that $\gamma_{1}\leq \gamma_{2}$ if and only if $\gamma_{1}(\alpha)\leq_{\alpha} \gamma_{2}(\alpha)$ for every $\alpha\in I(\eta)$. We observe that, if $\gamma_{1}\leq\gamma_{2}$ then $S(\gamma_{1})\subset S(\gamma_{2})$. By the definition of $(\Gamma_{\alpha}^{'},\leq_{\alpha})$ it is clear that $(\Gamma,\leq)$ is an up-directed partially ordered set.\\
For every $\gamma\in \Gamma$ we define $r_{\gamma}:X\rightarrow X$ as follows: let $(\alpha,x)\in X$ and $\beta_{\alpha}=\sup([0,\alpha]\cap S(\gamma))$
\begin{equation*}
r_{\gamma}(\alpha,x)=\begin{cases}
(\alpha,r_{\gamma(\alpha)}^{\alpha}(x)) \: &\mbox{if } \alpha \mbox{ is isolated and } \gamma(\alpha)\neq 0;\\
(\beta_{\alpha},\beta_{\alpha}) \: &\mbox{if }\alpha \mbox{ is isolated and } \gamma(\alpha)= 0 \mbox{ or } \alpha \mbox{ is limit}.
\end{cases}
\end{equation*}

\noindent
\underline{Claim}: for every $\gamma\in \Gamma$ the retraction $r_{\gamma}$ is a continuous mapping.\\
Let $\gamma\in\Gamma$ and let $\{(\alpha_{\lambda},x_{\lambda})\}_{\lambda\in\Lambda}$ be a net converging to $(\alpha,x)$. If $\alpha$ is isolated then $\{\alpha_{\lambda}\}_{\lambda\in\Lambda}$ is eventually constant and equal to $\alpha$;
\begin{itemize}
\item if $\alpha\in S(\gamma)$, by continuity of $r_{\gamma(\alpha)}^{\alpha}$, we have $\lim_{\lambda\in\Lambda}r_{\gamma}(\alpha_{\lambda},x_{\lambda})=r_{\gamma}(\alpha,x)$;
\item if $\alpha\notin S(\gamma)$ we have that $r_{\gamma}(\alpha_{\lambda},x_{\lambda})$ is eventually constant, hence it is clear that $\lim_{\lambda\in\Lambda}r_{\gamma}(\alpha_{\lambda},x_{\lambda})=r_{\gamma}(\alpha,x)$.
\end{itemize}
In the case of $\alpha$ limit we can suppose without loss of generality that $\alpha_{\lambda}\leq \alpha$ for every $\lambda\in\Lambda$. Two cases are possible:
\begin{itemize}
\item if $\sup([0,\alpha]\cap S(\gamma))=\alpha$, since $\alpha$ has countably cofinality, by definition of supremum there exists a sequence $\{\xi_{n}\}_{n\in\omega_{0}}\in [0,\alpha]\cap S(\gamma)$ such that $\xi_{n}$ is increasing and convergent to $\alpha$. For every $n\in\omega_{0}$, by definition of convergence of $\{\alpha_{\lambda}\}_{\lambda\in\Lambda}$ there exists $\lambda_{n}$ such that $\alpha_{\lambda}\in [\xi_{n},\alpha]$ for every $\lambda\geq \lambda_{n}$. Therefore $\lim_{\lambda\in\Lambda}r_{\gamma}(\alpha_{\lambda},x_{\lambda})=(\alpha,\alpha)=r_{\gamma}(\alpha,x)$.
\item if $\beta=\sup([0,\alpha]\cap S(\gamma))<\alpha$ then for a sufficiently large $\lambda$ we have $\beta<\alpha_{\lambda}\leq\alpha$; hence by definition of $r_{\gamma}$ we have $\lim\limits_{\lambda\in\Lambda} r_{\gamma}(\alpha_{\lambda},x_{\lambda})=r_{\gamma}(\alpha,x)=(\beta,\beta)$.
\end{itemize}
Therefore $r_{\gamma}$ is continuous for every $\gamma\in\Gamma$.\\

It remains to prove that $\{r_{\gamma}\}_{\gamma\in\Gamma}$ is a full retractional skeleton.
\begin{enumerate}[$(i)$]
\item for every $\gamma\in \Gamma$ the subspace $r_{\gamma}[X]$ is a closed subset of $X$, hence it is countably compact. Moreover, since it has a countable network, it is compact and metrizable.
\item Let $\gamma_{1},\gamma_{2}\in\Gamma$ such that $\gamma_{1}\leq \gamma_{2}$, then $S(\gamma_{1})\subset S(\gamma_{2})$. Let $(\alpha,x)\in X$
\begin{itemize}
\item if $\alpha$ is limit let $\beta_{\alpha}=\sup([0,\alpha]\cap S(\gamma_{1}))=\sup([0,\alpha]\cap S(\gamma_{1})\cap S(\gamma_{2}))$ and $\beta_{\alpha}^{'}=\sup([0,\alpha]\cap S(\gamma_{2}))$. We observe that $\beta_{\alpha}\leq\beta_{\alpha}^{'}\leq\alpha$ and that the intersection of the interval $(\beta_{\alpha},\beta_{\alpha}^{'})$ and $S(\gamma_{1})$ is empty. Hence:
\begin{equation*}
r_{\gamma_{1}}\circ r_{\gamma_{2}}(\alpha,\alpha)=r_{\gamma_{2}}\circ r_{\gamma_{1}}(\alpha,\alpha)=r_{\gamma_{1}}(\alpha,\alpha)=(\beta_{\alpha},\beta_{\alpha}).
\end{equation*}
\item if $\alpha$ is isolated three cases are possible:
\begin{enumerate}[(1)]
\item if $\gamma_{1}(\alpha)=0$ and $\gamma_{2}(\alpha)=0$, let $\beta_{\alpha}=\sup([0,\alpha]\cap S(\gamma_{1}))$ and $\beta_{\alpha}^{'}=\sup([0,\alpha]\cap S(\gamma_{2}))$, thus using the same observation of the previous point we have:
\begin{equation*}
r_{\gamma_{2}}(r_{\gamma_{1}}(\alpha,x))=r_{\gamma_{2}}(\beta_{\alpha},\beta_{\alpha})=(\beta_{\alpha},\beta_{\alpha})=r_{\gamma}(\alpha,x),
\end{equation*}
\begin{equation*}
r_{\gamma_{1}}(r_{\gamma_{2}}(\alpha,x))=r_{\gamma_{2}}(\beta_{\alpha}^{'},\beta_{\alpha}^{'})=(\beta_{\alpha},\beta_{\alpha})=r_{\gamma}(\alpha,x);
\end{equation*}
\item if $\gamma_{1}(\alpha)=0$ and $\gamma_{2}(\alpha)\neq 0$, let $\beta_{\alpha}=\sup([0,\alpha]\cap S(\gamma_{1}))$. Thus we have
\begin{equation*}
r_{\gamma_{1}}(\alpha,x)=(\beta_{\alpha},\beta_{\alpha}),
\end{equation*}
since $S(\gamma_{1})\subset S(\gamma_{2})$
\begin{equation*}
r_{\gamma_{2}}(r_{\gamma_{1}}(\alpha,x))=r_{\gamma_{2}}(\beta_{\alpha},\beta_{\alpha})=(\beta_{\alpha},\beta_{\alpha})
\end{equation*}
and at the end we have
\begin{equation*}
r_{\gamma_{1}}(r_{\gamma_{2}}(\alpha,x))=r_{\gamma_{1}}(\alpha,r_{\gamma_{2}(\alpha)}^{\alpha}(x))=(\beta_{\alpha},\beta_{\alpha});
\end{equation*}
\item if $\gamma_{1}(\alpha)\neq 0$ and $\gamma_{2}(\alpha)\neq 0$, since for every $\alpha\in I(\eta)$ the family $\{r_{s}^{\alpha}\}_{s\in\Gamma_{\alpha}}$ is a retractional skeleton on $X_{\alpha}$, we have
\begin{equation*}
\begin{split}
&r_{\gamma_{1}}\circ r_{\gamma_{2}}(\alpha,x)=(\alpha,r_{\gamma_{1}}^{\alpha}\circ r_{\gamma_{2}}^{\alpha}(x))=(\alpha,r_{\gamma_{1}}^{\alpha}(x))=r_{\gamma_{1}}(\alpha,x),\\
&r_{\gamma_{2}}\circ r_{\gamma_{1}}(\alpha,x)=(\alpha,r_{\gamma_{2}}^{\alpha}\circ r_{\gamma_{1}}^{\alpha}(x))=(\alpha,r_{\gamma_{1}}^{\alpha}(x))=r_{\gamma_{1}}(\alpha,x).\\
\end{split}
\end{equation*}
\end{enumerate}
\end{itemize}
\item Let $\gamma_{1}\leq\gamma_{2}\leq ...\leq \gamma_{n}\leq ...$ and define $\gamma(\alpha)=\sup_{n<\omega_{0}}\gamma_{n}(\alpha)$, for every $\alpha\in I(\eta)$. Since $S(\gamma_{n})\subset S(\gamma_{n+1})$ for every $n<\omega_{0}$ we have that $\bigcup_{n\in\omega_{0}}S(\gamma_{n})=S(\gamma)$ is countable, $0\in S(\gamma)$ and if $\alpha\in S(\gamma)$ there exists $n\in \omega_{0}$ such that $\alpha\in S(\gamma_{n})$, hence $[\alpha,\alpha+\omega_{0})\subset S(\gamma_{n})\subset S(\gamma)$, thus $\gamma\in\Gamma$. Moreover, let $(\alpha,x)\in X$ then
\begin{itemize}
\item if $\alpha\in S(\gamma_{n_{0}})$ for some $n_{0}<\omega_{0}$ then $\alpha \in S(\gamma_{n})$ for every $n>n_{0}$. Thus we have
\begin{equation*}
\begin{split}
\lim_{n\in\omega_{0}} r_{\gamma_{n}}(\alpha,x)&=\lim_{n>n_{0}}r_{\gamma_{n}}(\alpha,x)=\lim_{n>n_{0}}(\alpha,r_{\gamma_{n}(\alpha)}^{\alpha}(x))\\
&=(\alpha,r_{\gamma(\alpha)}^{\alpha}(x))=r_{\gamma}(\alpha,x).
\end{split}
\end{equation*}
\item if $\alpha\notin S(\gamma)$, then
\begin{equation*}
\begin{split}
\lim_{n\in\omega_{0}}r_{\gamma_{n}}(\alpha,x)&=\lim_{n\in\omega_{0}} (\sup([0,\alpha]\cap S(\gamma_{n})), \sup ([0,\alpha]\cap S(\gamma_{n})))\\&
=\sup_{n<\omega_{0}}(\sup([0,\alpha]\cap S(\gamma_{n})), \sup([0,\alpha]\cap S(\gamma_{n})))\\
&=(\sup([0,\alpha]\cap(\bigcup\limits_{n\in\omega_{0}}S(\gamma_{n}))),\sup([0,\alpha]\cap(\bigcup\limits_{n\in\omega_{0}}S(\gamma_{n}))))\\
&=r_{\gamma}(\alpha,x).
\end{split}
\end{equation*} 
\end{itemize}

\item For every $(\alpha,x)\in X$ there exists a $\gamma\in \Gamma$ such that $r_{\gamma}(\alpha,x)=(\alpha,x)$. In fact
\begin{itemize}
\item if $\alpha$ is isolated, since $\{r_{s}^{\alpha}\}_{s\in\Gamma_{\alpha}}$ is a full retractional skeleton on $X_{\alpha}$ there exists $s\in\Gamma_{\alpha}$ such that $r_{s}^{\alpha}(x)=x$. Hence there exists $\gamma \in \Gamma$ with $\gamma(\alpha)\geq s$ such that $r_{\gamma}(\alpha,x)=(\alpha,x)$;
\item if $\alpha$ is limit, there exists a sequence $\{\alpha_{n}\}_{n\in\omega_{0}}$ of isolated points such that $\alpha_{n}\nearrow \alpha$, moreover there exists a $\gamma\in\Gamma$ such that $\{\alpha_{n}\}_{n\in\omega_{0}}\subset S(\gamma)$. Hence, we have $r_{\gamma}(\alpha,\alpha)=(\alpha,\alpha)$.
\end{itemize}
Hence $\lim\limits_{\gamma\in\Gamma}r_{\gamma}(\alpha,x)=(\alpha,x).$
\end{enumerate}
Finally it is clear that $X=\bigcup\limits_{\gamma\in\Gamma}r_{\gamma}[X]$.
\end{proof} 

\begin{lem}\label{weaklyetasomme}
Let $\eta$ be an uncountable cofinality ordinal and $X_{\alpha}$ be a weakly non-commutative Corson countably compact space for every isolated $\alpha<\eta$. Let $X$ be the $[0,\eta+1)$-sum of $\{X_{\alpha}\}_{\alpha\in I(\eta)}$. Then $X$ is a weakly non-commutative Corson countably compact space.
\end{lem}

\begin{proof}
For every $\alpha\in I(\eta)$ there are a non-commutative Corson countably compact space $Y_{\alpha}$ and a continuous surjection $f_{\alpha}:Y_{\alpha}\to X_{\alpha}$.\\
Now, we are going to define a suitable family of countably compact spaces indexed on $I(\eta)$. Let $\alpha\in I(\eta)$, put $Z_{\alpha}=Y_{\alpha}$ if $\alpha<\omega_{0}$, $Z_{\alpha}=\{\alpha\}$ if $\alpha=\beta+1$ and $\beta$ is limit finally put $Z_{\alpha}=Y_{\beta}$ if $\alpha=\beta+1>\omega_{0}$ and $\beta$ is not limit.\\
Let $X$ be the $[0,\eta+1)$-sum of $\{X_{\alpha}\}_{\alpha\in I(\eta)}$ and $Y=Z\oplus\{(\eta+1,\eta+1)\}$ where $Z$ is the countably $[0,\eta)$-sum of $\{Z_{\alpha}\}_{\alpha\in I(\eta)}$. By Lemma \ref{etaSomme} and Lemma \ref{Operznoncomm} $Y$ is a non-commutative Corson countably compact space.\\
Finally, we define a mappings $f:Y\to X$ as follows:
\begin{equation*}
f(\alpha,x)=\begin{cases}
(\alpha,\alpha) \; &\mbox{if } \alpha \mbox{ is limit;}\\
(\beta,\beta) &\mbox{if } \alpha=\beta+1 \mbox{ and } \beta \mbox{ is limit;}\\
(\beta,f_{\beta}(x)) &\mbox{if } \alpha=\beta+1, \: \beta>\omega_{0} \mbox{ and } \beta \mbox{ is not limit;}\\
(\alpha,f_{\alpha}(x)) &\mbox{if } \alpha<\omega_{0}.
\end{cases}
\end{equation*}
Since every limit ordinal is covered by its successor and every $f_{\alpha}$ is surjective, we have that $f$ is surjective.\\
\underline{Claim:} $f$ is continuous.\\
Let $\{(\alpha_{\lambda},x_{\lambda})\}_{\lambda\in \Lambda}$ be a converging net to $(\alpha,x)\in Y$. Two cases are possible:
\begin{enumerate}[(1)]
\item $\alpha$ is not limit, then, since $Z_{\alpha}$ is clopen there exists a $\lambda_{0}\in \Lambda$ such that $x_{\lambda}\in Z_{\alpha}$ for every $\lambda\geq\lambda_{0}$. Let $\alpha=\beta+1$, we split in other three cases:
\begin{itemize}
\item $\beta$ is limit, by definition of the function $f$, we have $f(\alpha_{\lambda},x_{\lambda})=(\beta,\beta)$ for every $\lambda\geq\lambda_{0}$. Hence we have $\lim_{\lambda\in\Lambda}f(\alpha_{\lambda},x_{\lambda})=(\beta,\beta)=f(\alpha,x)$;
\item $\beta$ is not limit and $\beta>\omega_{0}$, by definition $f(\alpha_{\lambda},x_{\lambda})=(\beta,f_{\beta}(x_{\lambda}))$ for every $\lambda\geq\lambda_{0}$. Hence, by continuity of $f_{\beta}$ we have $\lim_{\lambda\in\Lambda}f(\alpha_{\lambda},x_{\lambda})=(\beta,f_{\beta}(x))=f(\alpha,x)$;
\item $\alpha<\omega_{0}$ by definition $f(\alpha_{\lambda},x_{\lambda})=(\alpha,f_{\alpha}(x_{\lambda}))$ for every $\lambda\geq\lambda_{0}$. Hence, by continuity of $f_{\alpha}$ we have $\lim_{\lambda\in\Lambda}f(\alpha_{\lambda},x_{\lambda})=(\alpha,f_{\alpha}(x))=f(\alpha,x)$.
\end{itemize}
\item If $\alpha$ is limit, the using definition of the function $f$ and the topology on $X$ we have $\lim_{\lambda\in\Lambda}f(\alpha_{\lambda},x_{\lambda})=(\alpha,\alpha)=f(\alpha,\alpha)$.
\end{enumerate}
Therefore $X$ is a weakly non-commutative Corson countably compact space. 
\end{proof}

The following Lemma is an analogue of \cite[Lemma 2.4]{Kalenda1}. It gives a characterization of non-commutative Corson countably compact spaces in the class of weakly non-commutative countably compact.

\begin{lem}\label{CarattWeakly}
Let $X$ be a weakly non-commutative Corson countably compact space. The following assertions are equivalent:
\begin{enumerate}
\item $X$ is non-commutative Corson.
\item $X$ is Fr\'{e}chet-Urysohn.
\item $X$ has countable tightness.
\end{enumerate}
\end{lem}

\begin{proof}
$(1\Rightarrow 2)$ By Proposition \ref{CaratCountably} $X$ is induced by a retractional skeleton in $\beta X$, hence by Theorem \ref{PropRetr}, it is a Fr\'{e}chet-Urysohn space.\\
$(2\Rightarrow 3)$ is trivial.\\
$(3\Rightarrow 1)$ Since $X$ is a weakly non-commutative Corson countably compact space, there are a non-commutative Corson countably compact space and a continuous onto mapping $f:Y\to X$. Let $F\subset Y$ be a closed subspace. We claim that $f(F)$ is closed in $X$. Let $x\in \overline{f(F)}$, by countably tightness of $X$ there exists a countable subset $C\subset f(F)$ such that $x\in \overline{C}$. Since $f$ is a onto mapping, there exists a countable set $D\subset F$ such that $f(D)=C$. Then $\overline{D}\subset F$ and, by Theorem \ref{PropRetr} $\overline{D}$ is a metrizable countably compact space, hence it is compact. Thus $f(\overline{D})$ is compact and contains $x$. Since $x\in f(\overline{D})\subset f(F)$, we deduce that $f(F)$ is closed.
Therefore, since $f$ is a continuous closed onto mapping, it is a quotient mapping. Hence, the assertion follows by Lemma \ref{Operznoncomm}. 
\end{proof}

\begin{thm}\label{teoOrdinali}
\begin{enumerate}[(1)]
\item The ordinal space $[0,\eta]$ is non-commutative Corson if and only if it is Corson if and only if $\eta<\omega_{1}$.\\
The ordinal space $[0,\eta]$ is a weakly non-commutative Corson for every ordinal $\eta$.
\item Let $X$ be a countably compact set of ordinals. Then $X$ is a weakly non-commutative Corson countably compact space.
\item Let $X$ be a countably compact set of ordinals. Then $X$ is a non-commutative Corson countably compact if and only if every ordinal of uncountable cofinality is isolated in $X$.
\end{enumerate}
\end{thm}
\begin{proof}
\begin{enumerate}[$(1)$]
\item The first part is trivial. If $\eta$ has uncountable cofinality, the assertion follows by Lemma \ref{weaklyetasomme}. Isolated and countable cofinality cases follow by the uncountable cofinality case and Lemma \ref{OperzWnoncomm}.

\item Let $X$ be a countably compact set of ordinals. Let $\theta=\sup(X) $, and let $Y$ be the closure of $X$ in $[0,\theta]$.\\
Since $Y$ is a well-ordered compact space it is homeomorphic to $[0,\eta]$ for some ordinal $\eta$. By previous point $Y$ is weakly non-commutative Corson. By Lemma \ref{OperzWnoncomm} $X$ is a countably closed subspace of $Y$, hence  $X$ is a weakly non-commutative Corson countably compact space.

\item It follows by previous point and Lemma \ref{CarattWeakly}.

\end{enumerate}
\end{proof}

Using the previous result and Lemma \ref{OperzWnoncomm} we observe that \cite[Example 3.8]{Cuth1} is an example of weakly non-commutative Corson countably compact space that is not a non-commutative Valdivia compact. We observe that, for every set $\Gamma$, the space $[0,1]^{\Gamma}$ is a Valdivia compact space, therefore it is a non-commutative Valdivia compact space. Moreover, there exists $\Gamma$ such that $[0,1]^{\Gamma}$ contains a subspace homeomorphic to a compact space that is not weakly non-commutative Corson. Finally, recalling that the class of weakly non-commutative Corson countably compact spaces is closed under countably closed subspaces (Lemma \ref{OperzWnoncomm}), we have an example of non-commutative Valdivia compact space that is not weakly non-commutative Corson compact space. Thus, the two classes are independent.\\
Now, we give the (weakly) non-commutative Valdivia version of Lemma \ref{etaSomme}. The proof follows straightforwardly by defining the same family of retractions of Lemma \ref{etaSomme}.

\begin{prop}\label{propOrdinal}
Let $\eta$ be an uncountable cofinality ordinal. Any $[0,\eta +1)$-sum of (weakly) non-commutative Valdivia compact is again a (weakly) non-commutative Valdivia compact.
\end{prop}

Next result gives a characterization of compact ordinal segments in the non-commutative setting.

\begin{thm}\label{equivOrdinal}
Let $\eta$ be an ordinal. Then the following hold
\begin{enumerate}[$(i)$]
\item $[0,\eta]$ is non-commutative Valdivia.
\item $[0,\eta]$ is weakly non-commutative Valdivia.
\item $[0,\eta]$ is weakly non-commutative Corson.
\end{enumerate}
\end{thm}
\begin{proof}
It trivially follows by Theorem \ref{teoOrdinali} and Proposition \ref{propOrdinal}.
\end{proof}

\section{Aleksandrov duplicates}

We recall the definition of Aleksandrov duplicate $AD(X)$ of the space $X$. It is the space $X\times\{0,1\}$ with the topology in which all points of $X\times\{1\}$ are isolated, and neighborhoods of $(x,0)$ are $(U\times \{0,1\})\setminus \{(x,1)\}$, for every $U$ neighborhood of $x\in X$. We denote by $\pi$ the projection from $AD(X)$ onto $X$.\\
It is well known and easy to check that if $X$ is compact (countably compact) then $AD(X)$ is compact (countably compact).\\
Now we give the main result of this section about the relations between retractional skeletons and Alexandrov duplicates of topological spaces.
\begin{thm}
\begin{enumerate}[(1)]
\item Let $X$ be a (weakly) non-commutative Corson countably compact space, then $AD(X)$ is a (weakly) non commutative compact space as well.
\item Let $K$ be a non-comutative Valdivia compact space and $D$ be an induced subspace such that $|K\setminus D|<\aleph_{0}$, then $AD(K)$ is a non-commutative Valdivia compact space.
\item Let $K$ be a compact space, if $AD(K)$ is a non-commutative Valdivia compact space then so is $K$.
\item Let $\eta$ be an ordinal, then $AD([0,\eta])$ is a non-commutative Valdivia compact space. 
\end{enumerate}
\end{thm}
\begin{proof}
\begin{enumerate}[(1)]
\item Let $X$ be a non-commutative Corson countably compact space, then the assertion follows easily by \cite[Theorem 3.1]{FerreiraRojas}.\\
Let $X$ be a weakly non-commutative Corson countably compact space, then there exists a continuous onto mapping $f:Y\to X$, where $Y$ is a non-commutative Corson countably compact space. Let $AD(X)$ and $AD(Y)$ be the Aleksandrov duplicates of $X$ and $Y$ respectively. Hence, by the previous point, $AD(Y)$ is non-commutative Corson countably compact.
Let $g:AD(Y)\to AD(X)$ be a map defined by $g(x,i)=(f(x),i)$, it is clearly onto. The continuity follows by the continuity of $f$ and the definition of the topology of Aleksandrov duplicate. Therefore $AD(X)$ is a weakly non-commutative Corson countably compact space.
\item Let $K$ be a non-commutative Valdivia compact space. By hypothesis there exists a retractional skeleton $\{r_{s}\}_{s\in\Gamma}$ such that $|K\setminus D|< \aleph_{0}$ and $D=\bigcup_{s\in\Gamma}r_{s}[K]$. We observe that the restriction $\{r_{s}\upharpoonright_{D}\}_{s\in\Gamma}$ is a full retractional skeleton on the countably compact space $D$. Hence using \cite[Proposition 2.7]{FerreiraRojas} there exists a full retractional skeleton on $D$, $\{r_{A}\, :\,A\in[D]^{\leq\omega_{0}}\}$ such that for every $A\in [D]^{\leq\omega_{0}}$ we have $r_{A}(x)=x$ for every $x\in A$. Now, using Theorem \ref{PropRetr} and Proposition \ref{CaratCountably}, we extend every $r_A\in$ $\{r_{A}\, :\,A\in[D]^{\leq\omega_{0}}\}$ to $K$ as $\{R_{A}\, :\,A\in[D]^{\leq\omega_{0}}\}$  such that $R_{A}\upharpoonright_{D}=r_{A}$.
For every $A\in[D\times \{0,1\}]^{\leq \omega_{0}}$ we define $\hat{R}_{A}:AD(K)\to AD(K)$ as follows
\begin{equation*}
\hat{R}_{A}(x,i)=
\begin{cases}
(R_{\pi(A)}(x),i) & \mbox{ if } x\in\pi(A);\\

(x,1) & \mbox{ if } x\in K \setminus D \mbox{ and } i=1;\\
(R_{\pi(A)}(x),0) & \mbox{ otherwise}.
\end{cases}
\end{equation*}
Using the same argument of \cite[Theorem 3.1]{FerreiraRojas} it is possible to prove that the family of retractions $\{\hat{R_{A}}\}_{A\in [D\times \{0,1\}]^{\leq \omega_{0}}}$ is a retractional skeleton on $AD(K)$.
\item We use the same idea as in \cite[Theorem 2.13]{Kalenda3}. Let $D$ an induced subspace of $AD(K)$, we will prove that $D\cap K\times\{0\}$ is dense in $K\times \{0\}$; then using Lemma \ref{SubspaceRetra} we have that $K$ is a non-commutative Valdivia compact space.\\
Let $U$ be a non-empty open subset of $K$. Since $K$ is completely regular there exists a non-empty open subspace $V\subset K$, such that $\overline{V}\subset U$. Two cases are possible:
\begin{enumerate}[(a)]
\item $V$ is finite. Every point of $V$ is a $G_{\delta}$ point of $AD(K)$, hence $V\times\{0\}\subset D$. Hence $U\cap D\neq\emptyset$.
\item $V$ is infinite. $V\times \{1\}\subset D$ is infinite and has a cluster point $(x,i)\in D$, by Theorem \ref{PropRetr}. By definition of Aleksandrov duplicate topology, $i$ must be equal to $0$ and $(x,0)\in\overline{V}\times\{0\} $.  Hence $U\cap D\neq\emptyset$.
\end{enumerate}
By above observation $K$ is a non-commutative Valdivia compact space.

\item Let $\eta$ be an ordinal. Let $\mathcal{A}$ be the family of all closed countable subsets $A$ of $[0,\eta]$ such that $0\in A$ and if $\alpha\in A$ is isolated then it is isolated in $[0,\eta]$. Then the family of mappings $\{r_{A}\}_{A\in\mathcal{A}}$, where $r_{A}:[0,\eta]\to[0,\eta]$ and $r_{A}(\alpha)=\max([0,\alpha]\cap A)$, is a retractional skeleton on $[0,\eta]$, see \cite{Kalenda4} or \cite{KubisMicha}.\\
Let $AD(K)$ be the Aleksandrov duplicate of $K=[0,\eta]$. Let us define
\begin{equation*}
\Gamma=\{(A,B)\in\mathcal{A}\times\mathcal[K]^{\leq\omega_{0}}:\, \forall \beta(\neq\eta)\in \,B,\,\beta+1 \in \, A  \}.
\end{equation*}
Define the following order on $\Gamma$: $(A_{1},B_{1})\leq (A_{2},B_{2})$ if and only if $A_{1}\subset A_{2}$ and $B_{1}\subset B_{2}$. This way $(\Gamma, \leq)$ is an up-directed partially ordered set.\\
For every $\gamma=(A,B)\in\Gamma$ we define $r_{\gamma}:AD(K)\to AD(K)$ as follows
\begin{equation*}
r_{\gamma}(x,i)=\begin{cases}
(r_{A}(x),0) & \mbox{ if } (x,i)\in K\times \{0\}\cup (K\times\{1\}\setminus B\times\{1\})\\
(x,i) & \mbox{ if } (x,i)\in B\times\{1\}
\end{cases}
\end{equation*}
\underline{Claim:} For every $\gamma\in \Gamma$, $r_{\gamma}$ is continuous. Let $\gamma=(A,B)\in \Gamma$ and let $\{(x_{\lambda},i_{\lambda})\}_{\lambda\in\Lambda}$ be a net converging to $(x,i)\in AD(K)$. We can suppose without loss of generality that $x_{\lambda}\leq x$ for every $\lambda \in \Lambda$. Now we study all possible cases
\begin{itemize}
\item if $i=1$ then $\{(x_{\lambda},i_{\lambda})\}_{\lambda\in\Lambda}$ is eventually constant. Hence it follows that $\lim\limits_{\lambda\in\Lambda}r_{\gamma}(x_{\lambda},i)=r_{\gamma}(x,i)$.
\item if $i=0$ then $x_{\lambda}\to x$ in $K$, two cases are possible:
\begin{enumerate}[(a)]
\item  suppose that $(x_{\lambda},i_{\lambda})\notin B\times\{1\}$ for sufficiently large $\lambda$. Then we conclude by continuity of $r_{A}$;
\item suppose that $(x_{\lambda},i_{\lambda})\in B\times\{1\}$ for a cofinal set $\Lambda_{1}\subset \Lambda$. Since $x_{\lambda}\nearrow x$ for $\lambda\in \Lambda_{1}$, we have $x_{\lambda}+1\nearrow x$ as well. Hence we have $x\in A$, then by definition $r_{\gamma}(x,0)=(x,0)=(r_{A}(x),0)$. Finally, since $r_{\gamma}(x_{\lambda},i_{\lambda})$ is equal to $(x_{\lambda},1)$ for $\lambda\in\Lambda_{1}$ and equal to $(r_{A}(x_{\lambda}),0)$ otherwise, we have that $\lim\limits_{\lambda\in\Lambda}r_{\gamma}(x_{\lambda},i_{\lambda})=(x,0)=(r_{A}(x),0)$, then we are done.
\end{enumerate}
\end{itemize}
This proves the claim. It remains to show that $\{r_{\gamma}\}_{\gamma\in\Gamma}$ is a retractional skeleton.
\begin{enumerate}[$(i)$]
\item Since $r_{\gamma}[AD(K)]$ is compact and countable, it is metrizable.
\item Let $\gamma_{1}=(A_{1},B_{1})$, $\gamma_{2}=(A_{2},B_{2})\in\Gamma$ such that $\gamma_{1}\leq\gamma_{2}$. Let $(x,i)\in AD(K)$. Then, if $i=0$ we have $r_{\gamma_{1}}(x,0)=(r_{A_{1}}(x),0)=(r_{A_{1}}\circ r_{A_{2}}(x),0)=r_{\gamma_{1}}\circ r_{\gamma_{2}}(x,0)$ and $r_{\gamma_{2}}\circ r_{\gamma_{1}}(x,0)=(r_{A_{2}}\circ r_{A_{1}}(x),0)=(r_{A_{1}}(x),0)=r_{\gamma_{1}}(x,0)$. If $i=1$ three cases are possible
\begin{enumerate}[(a)]
\item $x\in B_{1}$ then $x\in B_{2}$ as well and the equality is trivial.
\item $x\notin B_{1}$ and $x\in B_{2}$ then we have $r_{\gamma_{2}}\circ r_{\gamma_{1}}(x,1)=(r_{A_{2}}\circ r_{A_{1}}(x),0)=(r_{A_{1}}(x),0)=r_{\gamma_{1}}(x,1)$ and $r_{\gamma_{1}}\circ r_{\gamma_{2}}(x,1)=r_{\gamma_{1}}(x,1)$.
\item $x\notin B_{2}$, then $r_{\gamma_{1}}(x,1)=(r_{A_{1}}(x),0)=(r_{A_{1}}\circ r_{A_{2}}(x),0)=r_{\gamma_{1}}\circ r_{\gamma_{2}}(x,1)$ and $r_{\gamma_{2}}\circ r_{\gamma_{1}}(x,1)=(r_{A_{2}}\circ r_{A_{1}}(x),0)=(r_{A_{1}}(x),0)=r_{\gamma_{1}}(x,1)$.
\end{enumerate}
\item Let $\gamma_{1}\leq\gamma_{2}\leq...$ with $\gamma_{n}=(A_{n},B_{n})$ for every $n<\omega_{0}$ and $\gamma=\sup_{n<\omega_{0}}(A_{n},B_{n})=(\overline{\bigcup_{n<\omega_{0}}A_{n}},\bigcup_{n<\omega_{0}}B_{n})$.\\
We observe that 
\begin{itemize}
\item $|\overline{\bigcup_{n<\omega_{0}}A_{n}}|\leq\aleph_{0}$, 
\item $0\in\overline{\bigcup_{n<\omega_{0}}A_{n}}$
\item every isolated point of $\overline{\bigcup_{n<\omega_{0}}A_{n}}$ belongs to some $A_{k}$, hence it is isolated also in $K$.
\end{itemize}
Moreover $|\bigcup_{n<\omega_{0}}B_{n}|\leq \aleph_{0}$ and for every $\beta\in B$ there exists $k\in\omega_{0}$ such that $\beta\in B_k$, hence $\beta+1\in A_{k}\subset\overline{\bigcup_{n<\omega_{0}}A_{n}}$. Therefore $\gamma\in \Gamma$.\\
Let $(x,i)\in AD(K)$, if $i=0$ or $i=1$ and $x\notin B_{n}$ for every $n\in\omega_{0}$ we have
\begin{equation*}
\lim\limits_{n\in\omega_{0}}r_{\gamma_{n}}(x,i)=\lim\limits_{n\in\omega_{0}}(r_{A_{n}}x,0)=r_{\gamma}(x,i).
\end{equation*}  
If $i=1$ and $x\in B_{n_{0}}$ for some $n_{0}<\omega_{0}$ then
\begin{equation*}
\lim\limits_{n\in\omega_{0}}r_{\gamma_{n}}(x,1)=\lim\limits_{n_{0}<n}(x,1)=r_{\gamma}(x,1).
\end{equation*}
\item Let $(x,i)\in AD(K)$, if $i=0$  then $$\lim_{\gamma\in\Gamma}r_{\gamma}(x,0)=\lim_{A\in\mathcal{A}}(r_{A}(x),0)=(x,0).$$ If $i=1$ there exists a $B\in [K]^{\leq\omega_{0}}$ such that $x\in B$ then $$\lim_{\gamma\in\Gamma}r_{\gamma}(x,1)=(x,1).$$
\end{enumerate}
Hence $\{r_{\gamma}\}_{\gamma\in\Gamma}$ is a retractional skeleton on $AD(K)$.

\end{enumerate}
This completes the proof.
\end{proof}

\section{Associated Banach spaces}

In this section we generalize Section 4 of \cite{Kalenda1} to the non-commutative case. Though the ideas, of following proofs, are the same of \cite{Kalenda1} we prefer to give every result because they are interesting in the non-separable Banach space theory setting. 

\begin{prop}\label{Banach1porjEquiv}
Let $X$ be a Banach space, then the following assertions are equivalent.
\begin{enumerate}[($i$)]
\item $X$ is linearly isometric to a subspace of a Banach space $Y$, that has a $1$-Projectional skeleton.
\item $(B_{X^{*}},w^{*})$ contains a dense convex symmetric weakly non-commutative Corson countably compact space.
\item $(B_{X^{*}},w^{*})$ is a weakly non-commutative Valdivia space.
\end{enumerate}
\end{prop}
\begin{proof}
$(i)\Rightarrow(ii)$ Let $i$ be the isometric injection of $X$ into $Y$ and $i^{*}$ be the adjoint surjection of $Y^{*}$ onto $X^{*}$.\\
By \cite[Proposition 3.14]{Cuth1} there exists a convex symmetric set $R$ induced by a retractional skeleton in $(B_{Y^{*}},w^{*})$. Hence by Theorem \ref{PropRetr} $R$ is a dense non-commutative Corson countably compact space.\\
Since $i^{*}$ is a linear $w^{*}$-$w^{*}$-continuous mapping and $i^{*}(B_{Y^{*}})=B_{X^{*}}$, we have that $i^{*}(R)$ is a convex symmetric dense weakly non-commutative Corson countably compact space.\\
$(ii)\Rightarrow(iii)$ Easily follows by definition of weakly non-commutative Valdivia compact space.\\
$(iii)\Rightarrow(i)$ Suppose $(B_{X^{*}},w^{*})$ is weakly non-commutative Valdivia. Then there exist a non-commutative Valdivia compact space $K$ and a continuous onto mapping $T:K\to(B_{X^{*}},w^{*})$. By \cite[Proposition 3.15]{Cuth1} $C(K)$ has a $1$-Projectional skeleton. Moreover, using the adjoint map of $T$ we have that $C(B_{X^{*}},w^{*})$ is an isometric subspace of $C(K)$. Observing that, by Hahn-Banach extension theorem, we have $X\hookrightarrow C(B_{X^{*}},w^{*})$, $X\hookrightarrow C(K)$ isometrically.
\end{proof}

\begin{defn}
Let $X$ be a Banach space, we will call it \textit{non-commutative weakly WLD} if and only if $(B_{X^{*}},w^{*})$ is a weakly non-commutative Corson compact.
\end{defn}

\begin{prop}
The class of non-commutative weakly WLD Banach spaces is closed under isomorphisms, subspaces and quotients.
\end{prop}

\begin{proof}
Let $X$ be a non-commutative weakly WLD Banach space and the map $T:Y\to X$ be an isomorphism. Using Lemma \ref{OperzWnoncomm} and observing that $(T^{*})^{-1}(B_{Y^{*}})$ is a weak$^{*}$-closed subspace of $\lambda B_{X^{*}}$, for some $\lambda \in \R^{+}$, we have that $B_{Y^{*}}$ is a weakly non-commutative Corson compact space. Therefore $Y$ is a non-commutative weakly WLD Banach space.\\
Let $X$ be a non-commutative weakly WLD Banach space and $Y$ be a closed subspace of $X$. Let $i:Y\to X$ be the canonical embedding then $i^{*}(B_{X^{*}})= B_{Y^{*}}$. Using Lemma \ref{OperzWnoncomm} we have that $B_{Y^{*}}$ is a weakly non-commutative Corson compact. Therefore $Y$ is a non-commutative weakly WLD Banach space.\\
Let $X$ be a non-commutative weakly WLD Banach space and $Y$ be a closed subspace. Since $(X/Y)^{*}$ is canonically isometric and weak$^{*}$ homeomorphic to $Y^{\perp}$ and $Y^{\perp}$ is a weak$^{*}$-closed subspace of $X^{*}$, we have by Lemma \ref{OperzWnoncomm} that $B_{(X/Y)^{*}}$ is weakly non-commutative Corson. Therefore $X/Y$ is a non-commutative weakly WLD Banach space.
\end{proof}

\begin{prop}\label{BanachNo1proj}
Let $X$ be a Banach space such that $X^{*}$ contains a convex weak$^{*}$ -compact subset that is not weakly non-commutative Valdivia. Then there is an equivalent norm on $X\times \R$ such that it cannot be linearly isometric to any subspace of $Y$, where $Y$ is any Banach space with 1-Projectional skeleton.
\end{prop}
\begin{proof}
Let $K\subset X^{*}$ be a convex weak$^{*}$ compact that is not weakly non-commutative Valdivia. Put
\begin{equation*}
B=\mbox{conv}((\frac{1}{2}B_{X^{*}}\times [-\frac{1}{2},\frac{1}{2}])\cup(K\times \{1\})\cup(-K\times \{1\}) ).
\end{equation*}
Since $(\frac{1}{2}B_{X^{*}}\times [-\frac{1}{2},\frac{1}{2}])\subset B\subset (B_{X^{*}}\times [-1,1])$, we have that $B$ is a dual unit ball of $X\times \R$. Furthermore, since $K\times \{1\}=\bigcap_{n\in\omega_{0}}K\times(1-\frac{1}{n},1] $, it is a weak$^{*}$ closed weak$^{*}$ G$_{\delta}$ subset of $B$. Suppose by contradiction that $B$ is weakly non-commutative Valdivia. By Proposition \ref{operazWNVald} $K$ must be a weakly non-commutative Valdivia compact space, which is a contradiction. Then the statement follows by Proposition \ref{Banach1porjEquiv}.
\end{proof}

We observe that it is possible to prove the previous result in the complex setting, we define
\begin{equation*}
B=\mbox{conv}( \frac{1}{2}B_{X^{*}}\times \overline{ D(0,1/2)}\cup \overline{\mbox{conv}\bigcup_{|\alpha|=1} \{\alpha K\times \{\alpha\}\}}^{w^{*}} ),
\end{equation*}
where $D(0,1/2)$ is the complex disk with centre in the origin and radius equal to $1/2$. Then the proof follows using the same construction of \cite[Theorem 4]{Kalenda5}.

\begin{cor}
Let $K$ be a compact space such that there exists a closed subset $L\subset K$ with a dense set of (relatively) $G_{\delta}$ points which is not weakly non-commutative Valdivia. Then there exists an equivalent norm on $C(K)$ such that it cannot be linearly isometric to any subspace of $Y$, where $Y$ is any Banach space with 1-Projectional skeleton.
\end{cor}
\begin{proof}
By Theorem \ref{TeoremaMisurediradon} $P(L)$ is not weakly non-commutative Valdivia. Since $P(L)$ can be identified with a convex weak$ ^{*} $ compact subset of $P(K)$. Then the statement follows by Proposition \ref{BanachNo1proj}.
\end{proof}

\begin{prop}
Let $X,Y$ be Banach spaces. Suppose that $Y$ has a $1$-projectional skeleton, the norm on $X$ is G\^{a}teaux smooth and $X$ is isometric to some subspace of $Y$. Then $X$ is non-commutative weakly WLD. 
\end{prop}

\begin{proof}
Let $D$ be a dense weakly non-commutative Corson countably compact subspace of $B_{X^{*}}$. Let $x^{*}\in S_{X^{*}}$ a functional that attains its norm at some point $x_{0}\in S_{X}$. Let $\{x_{n}\}_{n\in\omega_{0}}\subset S_{X}$ such that $x^{*}(x_{n})>(1-1/n)$ for every $n\in\omega_{0}$ and $x_{n}\to x$ in norm. Then $x^{*}\in(\bigcap\limits_{n\in\omega_{0}}[\{y^{*}\in X^{*}:\, y^{*}(x_{n})>1-1/n\}\cap B_{X^{*}}])$. Since the norm on $X$ is G\^{a}teaux differentiable in $x_{0}$ then 
\begin{equation*}
x^{*}=(\bigcap\limits_{n\in\omega_{0}}[\{y^{*}\in X^{*}:\, y^{*}(x_{n})>1-1/n\}\cap B_{X^{*}}]),
\end{equation*}
therefore $x^{*}$ is a weak$^{*}$ $G_{\delta}$ point of $B_{X^{*}}$. Hence it belongs to $D$. By the Bishop-Phelps theorem the subset of norm-attaining functionals is norm dense in $S_{X^{*}}$. Since $D$ is closed with respect to limits of sequences we have that $S_{X^{*}}\subset D$. Using the corollary of Josefson-Nissenzweig theorem we have that $B_{X^{*}}\subset D$. Hence we are done.
\end{proof}

\begin{prop}
Let $K$ be a weakly non-commutative Corson compact space with property $(M)$. If $L$ is a compact space such that $C(L)$ is isomorphic to $C(K)$, then $L$ is weakly non-commutative Corson.
\end{prop}
\begin{proof}
Since $K$ is a weakly non-commutative Corson compact space with property $(M)$, using Theorem \ref{TeoPropM} we have that $C(K,\C)$ is non-commutative weakly WLD. Since the class of non-commutative weakly WLD spaces is closed under isomorphisms we have that $C(L,\C)$ is non-commutative weakly WLD. Therefore by definition of non-commutative weakly WLD we have that $(B_{C(L,\C)^{*}},w^{*})$ is a weakly non-commutative Corson compact space. Finally, since $L$ is a closed subset of $(B_{C(L,\C)^{*}},w^{*})$, using Lemma \ref{OperzWnoncomm}, we have that $L$ is weakly non-commutative Corson compact as well. The proof holds for the real case as well.

\end{proof}

Observing that for every ordinal number $\eta$ the topological space $[0,\eta]$ has the property $(M)$ we have the following result.

\begin{cor}
Let $L$ be a compact space. Suppose that $C(L)$ is isomorphic to $C([0,\eta])$, for some ordinal number $\eta$. Then $L$ is a weakly non-commutative Corson compact space.
\end{cor}

\noindent
\textbf{Acknowledgements:} The author is grateful to Ond\v{r}ej Kalenda for suggesting the topic and for the many helpful discussions about it.

\noindent
Jacopo Somaglia\\
Dipartimento di Matematica\\
Universit\`{a} degli studi di Milano\\
Via C. Saldini, 50\\
20133 Milano MI, Italy\\
jacopo.somaglia@unimi.it\\
ph: ++39 02 503 16167

\end{document}